\documentclass[12pt,leqno]{article}

\usepackage{amssymb,amsmath,amsthm}
\usepackage[all]{xy}
\usepackage{tam}
\numberwithin{equation}{section}
\begin{document}
\title{$\shd$-modules and complex foliations}
\author{Hamidou Dathe}
\maketitle

\begin{abstract}
Consider a complex analytic manifold $X$ and  a coherent Lie subalgebra $\shi$ of
the Lie algebra of complex vector fields on $X$. By using a natural
$\shd_X$-module $\shm_\shi$ naturally associated to $\shi$ and the
ring (in the derived sense) $\rhom[\shd_X](\shm_\shi,\shm_\shi)$, we associate
integers which measure the irregularity of the foliation associated
with $\shi$.
\end{abstract}

%\thanks{Mathematics Subject Classification: 53D55, 46L65, 32C38}
%\tableofcontents
\section{Introduction}

The idea of using $\shd$-module theory in the study of foliations is very natural and appeared in particular in
\cite{Su90,ESL09}. More precisely, consider a complex manifold $X$ and a coherent Lie subalgebra $\shi$ of the sheaf $\Theta_X$ of tangent vectors. To this ideal is naturally associated associated the coherent left $\shd$-module $\shm_\shi=\shd_X/\shd_X\cdot\shi$ already considered in \cite{Su90}
. The new idea of this paper is to consider the ring (in the derived sense)  $\shd_\shi\eqdot\rhom[\shd_X](\shm_\shi,\shm_\shi)$ and
$\shd_\shi^0\eqdot H^0(\shd_\shi)$.
%\item
We denote by $\shd$-$\irr(\shi)$ the increasing sequence consisting of the
 integers $k$ such that $H^k(\shd_\shi)\not=0$ and call it the $\shd$-irregularity 
 of the foliation. These integers $k\in\N$ for which  the cohomology of this ring is not $0$ give invariant which measure, in some sense, the irregularity of the foliation. We call the biggest of these integers the D-irregularity of the foliation.
With some hypotheses, already considered in~\cite{Su90}, we compute the D-irregularity and calculate  $\shd$-$\irr(\shi)$ for some examples of foliation. We also give geometrically interpretation  of $\shd$-$\irr(\shi)$.

\vspace{0.4ex}\noindent
{\bf Acknowledgments}\\
The author  warmly thanks  Pierre Schapira  for proposing him this subject and for his advices during the preparation of the manuscript.

\section{Foliations}
In all this paper, $X$ denotes a complex  manifold
of complex dimension $d_X$. One denotes by
$\sho_X$ the structure sheaf, by $\Theta_X$ the sheaf of holomorphic vector
fields, by  $\Omega^1_X$ the sheaf of holomorphic $1$-forms and by
$\shd_X$ the sheaf of holomorphic differential operators.

\subsubsection*{Vector fields and $1$-forms}
Consider first a locally free $\sho_X$-module of finite rank $\shl$
and set
\eqn
&&\shl^*=\hom[\sho_X](\shl,\sho_X).
\eneqn
One denotes by $\langle\scbul,\scbul\rangle$ the pairing
$(\shl,\shl^*)\to\sho_X$. One uses the same notations when
interwinning $\shl$ and $\shl^*$.

Let  $\shi$ be an $\sho_X$-submodule  of $\shl$. Recall that $\shi$
is coherent if and only if it is locally finitely generated, that is, if
there exists  a locally free $\sho_X$-module of finite rank $\shk$ and
an $\sho_X$-linear map $\psi\cl\shk\to\shl$ such that $\shi=\im\psi$.

One defines the
orthogonal $\shi^\perp$ in $\shl^*$  to $\shi$ by
\eqn
&&\omega\in\shi^\perp\Leftrightarrow\langle\omega,v\rangle=0
\mbox{ for all }v\in\shi.
\eneqn
More precisely, $\shi^\perp$ is the sheaf associated with the presheaf
$U\mapsto \shi(U)^{\perp}$.

\begin{lemma}\label{le:ortho1}
Let  $\shi$ be a coherent $\sho_X$-submodule  of $\shl$.
Then $\shi^\perp$ is coherent and $\shi\isoto \shi^{\perp,\perp}$.
\end{lemma}
\begin{proof}
Consider $\psi\cl\shk\to\shl$ as above such that $\shi=\im\psi$
and denote by $\psi^*\cl\shl^*\to\shk^*$ the dual map.
Then
\eqn
&&(\im\psi)^\perp\simeq \ker\psi^*,\\
&&\ker(\psi^*)^\perp\simeq \im\psi,
\eneqn
 and the result follows since $\mdcoh[\sho_X]$ is abelian.
\end{proof}

In the sequel, we shall apply Lemma~\ref{le:ortho1} to the case where
$\shl=\Theta_X$  and thus $\shl^*=\Omega_X^1$.
Note that  the duality is given by
\eqn
&&\langle\sum_ia_i(x)df_i, v\rangle=\sum_ia_i(x)v(f_i),\quad v\in\Theta_X.
\eneqn
The next result is well-known.
\begin{proposition}
Let $\shi$ be a coherent submodule of $\Theta_X$. Then $\shi$
is a Lie subalgebra of $\Theta_X$ if and only if the coherent
submodule $\shi^\perp$ of $\Omega_X^1$ satisfies $d {\shi}^{\perp}\subset {\shi}^{\perp}\wedge\shi^\perp$.
\end{proposition}
\begin{proof}
For $\omega$ a section of $\Omega_X^1$ and $v_1,v_2$ two sections of $\Theta_X$, recall the formula
\eqn
&&d\omega(v_1,v_2)=v_1\cdot \langle\omega,v_2\rangle -v_2\cdot \langle\omega,v_1 \rangle -
\langle\omega,[v_1,v_2]\rangle.
\eneqn
Choosing $\omega$ in $\shi^\perp$, we find that
$d\omega(v_1,v_2)=0$ if and only if $\langle\omega,[v_1,v_2]\rangle=0$. Hence\\
(i) if $\shi$ is a Lie subalgebra, then $d\omega\in {\shi}^{\perp}\wedge\shi^\perp$, \\
and conversely:\\
(ii) if $d {\shi}^{\perp}\subset {\shi}^{\perp}\wedge\shi^\perp$, then $d\omega(v_1,v_2)=0$ for any
$\omega$ in  ${\shi}^{\perp}$, hence $[v_1,v_2]$ belongs to $\shi^{\perp,\perp}\simeq\shi$.
\end{proof}
\begin{example}
Consider a  complex Poisson manifold $X$, that is, a complex manifold $X$ endowed
with a bracket $\{\scbul,\scbul\}\cl\sho_X\times\sho_X\to\sho_X$
satisfying the Jacobi identities. One defines~\cite{ESL09} the Lie
sub-algebra $\shi$ of $\Theta_X$ as the ideal generated by the
derivations $\{f,\scbul\};f\in\sho_X$.
\end{example}

\subsubsection*{Foliations}
We consider now an  $\sho_X$-submodule $\shi$  of $\Theta_X$ and we assume that
\eq\label{hyp:1}
\left\{\parbox{60ex}{
(i) $\shi$ is $\sho_X$-coherent,\\
(ii) $\shi$ is a Lie subalgebra of $\Theta_X$.
}\right.
\eneq

We call $\shi$ a singular foliation of $X$.

For $x\in X$, we denote by $\shi(x)$ the subspace of $T_xX=\Theta_X(x)$
generated by the the germs of sections of $\shi$ at $x$.
\begin{definition}
One sets
\eqn
&&\left\{\parbox{60ex}{
$\rk(\shi)=\sup_{x\in X}\dim\shi(x)$,\\
$\cork(\shi)=\inf_{x\in X}\dim\shi(x)$,\\
$\irr(\shi)=\rk(\shi)-\cork(\shi)$.
}\right.\eneqn
One says that $\rk(\shi)$ is the rank of $\shi$, $\cork(\shi)$  the
corank of $\shi$ and $\irr(\shi)$ the irregulatity of $\shi$.
\end{definition}
We set
\eqn
&&X_j=\{x\in X;\dim\shi(x)=\rk(\shi)-j\}.
\eneqn
Hence, the $X_j's$ are locally closed complex analytic submanifolds
and one gets a stratification
\eqn
&&X=\bigsqcup_{j=0}^{\irr(\shi)}X_j.
\eneqn
Note that $X_0$ is an open dense subset of $X$,  $\shi$ is locally
free on $X_0$ and
$X\setminus X_0$ is a  closed complex analytic subset of
$X$ of codimension at least $1$.

One says that the foliation is regular if $\irr(\shi)=0$, that is,
if $X_0=X$.

\begin{definition}
Let $\Sigma\subset X$ be an embedded submanifold of $X$. One says that
$\Sigma$ is a leaf of $\shi$ if for any $x\in\Sigma$,
$T_x\Sigma=\shi(x)$.
\end{definition}

\begin{remark}
(i) The Frobenius theorem asserts that
each $X_j$ admits  a foliation by complex leaves of dimension $\rk(\shi)-j$.

\noindent
(ii) There is a theorem by Nagano~\cite{Na66} which asserts that $X$ is a unique disjoint
union of connected leaves. This result is false in the $C^\infty$-setting as shown by the following example, due to Nagano.
Let $X=\R^2$ endowed with coordinates $(x,y)$ and consider the two vector fields
on $\R^2$ given by $\partial_x,f(x)\partial_y$ and the left ideal $\shj$ they generate. Assume that $f(x)\not=0$ for $x\not=0$ and $f$ has a zero of infinite order at $x=0$. Then $\R^2=\{x\not=0\}\cup\{x=0\}$ and the set $\{x=0\}$ is not a leaf.

\noindent
(iii) The Camacho-Sad theorem asserts that in dimension $2$, each
point of $X$ belongs to the closure of a leaf of $X_0$.
\end{remark}

\section{Links with $\shd$-modules}
The idea of associating a $\shd$-module to a foliation is not new. See in particular~\cite{Su90} and for a systematic approach in the algebraic case, see~\cite{ESL09}.

We refer to Kashiwara~\cite{Ka03} for $\shd$-module theory.

Let $\shi$ satisfying~\eqref{hyp:1}. We denote by $\tw\shi$ the left
ideal of $\shd_X$ generated by $\shi$, that
is, $\tw\shi=\shd_X\cdot\shi$.
We denote by
$\shm_\shi\eqdot\shd_X/\tw\shi$ the associated $\shd_X$-module.
Since $\shi$ is $\sho_X$-coherent, the ideal $\tw\shi$ is locally
finitely generated, hence coherent in $\shd_X$ and therefore $\shm_\shi$
is coherent.
Such a module has already been considered and studied in~\cite{ESL09}.

Recall that to a coherent $\shd_X$-module $\shm$, one associates its
characteritic variety $\chv(\shm)$, a closed co-isotropic
$\C^*$-conic analytic subset of the cotangent bundle $T^*X$.
Recall a definition of~\cite[p.19]{Ka03}
\begin{definition}
Let $B\eqdot \{P_1,\dots,P_{r}\}$ be a system of generators of a left ideal $\shi$ of $\shd_X$.
One says that $B$ is an involutive  system of generators of $\shi$ if

$\gr B\eqdot \{\sigma(P_1),\dots,\sigma(P_r)\}$ is a system of generators of the graded ideal $\gr \shi$.
\end{definition}
For such an  involutive  system of generators, we evidently have
\eqn
&&\chv(\shd_X/\shi)=\{(x;\xi)\in T^*X; \sigma(P_j)(x;\xi)=0, \mbox{ for all }1\leq j\leq r\}.
\eneqn

We shall consider the hypotheses
\eq\label{hyp:2}
\left\{\parbox{60ex}{
there locally exist $v_1,\dots,v_r\in\shi$ such that
the family $(v_1,\dots,v_r)$ generates $\shi$,  $[v_i,v_j]=0$ for $1\leq i,j\leq r$ and the sequence of symbols
$\{\sigma(v_1),\dots,\sigma(v_r)\}$ is a regular sequence in $\sho_{T^*X}$,
}\right.
\eneq
and
\eq\label{hyp:3}
\left\{\parbox{60ex}{
there locally exist $v_1,\dots,v_r\in\shi$ such that
the family $(v_1,\dots,v_r)$ generates $\shi$,  $[v_i,v_j]=0$ for $1\leq i,j\leq r$ and the sequence $\{v_1,\dots,v_r\}$ is a regular sequence in $\shd_X$.
}\right.
\eneq

Note that the hypothesis~\eqref{hyp:2} has already been considered and studied in~\cite{Su90}.

In order to compare hypotheses~\eqref{hyp:2} and~\eqref{hyp:3}, recall some results from~\cite{Ka03}.
Consider a complex of filtered $\shd_X$-modules
\eq\label{eq:cplD}
&&\shm_1\to\shm_2\to\shm_3
\eneq
and the associated complex of graded modules
\eq\label{eq:cplgrD}
&&\gr\shm_1\to\gr\shm_2\to\gr\shm_3.
\eneq
By~\cite[Prop.A.17]{Ka03}, if \eqref{eq:cplgrD} is exact, then \eqref{eq:cplD} is filtered exact.

\begin{lemma}\label{le:regseq}
Hypothesis~\eqref{hyp:2} implies hypothesis~\eqref{hyp:3}.
\end{lemma}
\begin{proof}
For $1\leq d< r$ consider the assertion:
\eq\label{eq:cplD1}
&&v_{d+1}\mbox{ acting on } \shd_X/\sum_{j=1}^d\shd_X\cdot v_j\mbox{ is injective}
\eneq
and  the assertion
\eq\label{eq:cplgrD1}
&&\sigma(v_{d+1})\mbox{ acting on } \sho_{T^*X}/\sum_{j=1}^d\sho_{T^*X}\cdot \sigma(v_j) \mbox{ is injective}.
\eneq
It follows from the results mentioned above that \eqref{eq:cplgrD1} implies \eqref{eq:cplD1}.
\end{proof}
We do not know if  hypotheses~\eqref{hyp:2} and~\eqref{hyp:3} are equivalent.

\begin{proposition}\label{pro:charm}
Let $\shi$ satisfying~\eqref{hyp:1}  and~\eqref{hyp:2} and let $\shm_\shi$ be the
associated coherent $\shd_X$-module. Then
\bnum
\item
$(v_1,\dots,v_r)$ is an involutive  system of generators of $\tw\shi$,
\item
$\chv(\shm_\shi)=\{(x;\xi)\in T^*X; v(x;\xi)=0\mbox{ for all }v\in\shi\}$,
\item
$r=\rk(\shi)=\codim \chv(\shm_\shi)$,
\item
the projective dimension of $\shm_\shi$ is equal to $\rk(\shi)$. In other words,
$\ext[\shd_X]{j}(\shm_\shi,\shd_X)=0$ for $j\not=\rk(\shi)$.
\enum
\end{proposition}
Note that the assertion on $\chv(\shm)$  was already obtained in~\cite[Cor.~4.13]{Su90}.
\begin{proof}
(i) The sequence of symbols being regular, the variety $\bigcap_{j-1}^r\opb{\sigma}(v_j)(0)$ has codimension $r$.
Then apply~\cite[Prop.~2.12]{Ka03}.

\vspace{0.3ex}\noindent
(ii)  follows from (i).

\vspace{0.3ex}\noindent
(iii) It follows from (i) and (ii) that the variety $\chv(\shm)$ has codimension $r$. Hence, the leaves of this involutive manifold have dimension $r$ and this dimension is the rank of $\shi$.

\vspace{0.3ex}\noindent
(iv) Consider the Koszul complex $K^\scbul(\shd_X,\{v_1,\dots,v_r\})$.
The cohomology of this complex is concentrated in degree $r$ by~\eqref{hyp:3} (which follows
from Lemma~\ref{le:regseq}) and is isomorphic to
$\shm_\shi$.
Therefore $\shm_\shi$ admits a projective resolution of length $\leq r$ and
the projective dimension of $\shm_\shi$ is $\leq r$. On the other hand,
$\ext[\shd_X]{j}(\shm_\shi,\shd_X)=0$ for $j>\rk(\shi)$ by ~\cite[Th.~2.19]{Ka03}.
\end{proof}

\begin{definition}
\banum
\item
We set $\shd_\shi\eqdot\rhom[\shd_X](\shm_\shi,\shm_\shi)$ and
$\shd_\shi^0\eqdot H^0(\shd_\shi)$.
\item
We denote by $\shd$-$\irr(\shi)$ the increasing sequence consisting of the
 integers $k$ such that $H^k(\shd_\shi)\not=0$
and we set D-$\irr(\shi)$=$\sup(\shd\text{-}\irr(\shi))$.
\eanum
\end{definition}

\begin{remark}
(i) The sequence $\shd$-$\irr(\shi)$ and the integer D-$\irr(\shi)$ are invariants of the foliation.

\noindent
(ii) Note that $\shd_\shi^0$ is a ring
and the restriction of $\shd_\shi$ to $X_0$ is concentrated
in degree $0$.

\noindent
(iii) A similar ring to $\shd_\shi^0$ has also been constructed ``at hands''
in the real case for smooth foliations in~\cite{PRW10}.

\noindent
(iv) The sheaf $\hom[\shd_X](\shm_\shi,\sho_X )$  is the sheaf of holomorphic functions that are constant along the leaves of $\shi$.
\end{remark}

\begin{theorem}\label{th:1}
Assume~\eqref{hyp:1} and~\eqref{hyp:2}. Then
$\rm D$-$\irr(\shi)=\irr(\shi)$.
\end{theorem}
\begin{proof}
(A) First we prove the inequality $\rm D$-$\irr(\shi)\leq\irr(\shi)$.

\vspace{0.3ex}\noindent
(A)--(i) Assume first that $\cork(\shi)=0$, that is, $\irr(\shi)=\rk(\shi)$. In this case the result follows from
Proposition~\ref{pro:charm}. Indeed, $\ext[\shd_X]{k}(\shm_\shi,\shn)\simeq0$ for all $k>\rk(\shi)$ and all $\shd_X$-module $\shn$.

\vspace{0.3ex}\noindent
(A)--(ii) Assume $\cork(\shi)>0$. Let $(v_1,\dots,v_r)$ be as in Hypothesis~\ref{hyp:3}. We may choose a local coordinate system $(x_1,\dots,x_n)$ on $X$ such that $v_1=\partial_1$. Since $[v_1,v_j]=0$,  the $v_j$'s do not depend on $x_1$ and of course, we may also assume that they do not depend on $\partial_1$. (If $v_j=w_j+a(x)\partial_1$, we may replace
$v_j$ with $w_j$ keeping the same hypotheses.)

\vspace{0.3ex}\noindent
(A)--(iii) Arguing by induction, we may assume that
$$(v_1,\dots,v_r)=(\partial_1,\dots,\partial_s,v_{s+1},\dots v_r)$$
where $s=\cork(\shi)$ and the $v_j$'s ($j=s+1,\dots,r$) depend neither on $(x_1,\dots,x_s)$ nor on
 $(\partial_1,\dots,\partial_s)$.

 Let $X=X_1\times X_2$ where $X_1=\C^s$ and $X_2=\C^{n-s}$.
 Let $\shi_j$ denote the ideal of $\shd_{X_j}$ generated by
  $(\partial_1,\dots,\partial_s)$ in case $j=1$ and  by $(v_{s+1},\dots,v_r)$  in case $j=2$.
Set
\eqn
&&\tw\shi_j\eqdot\shd_{X_j}\cdot\shi_j,\quad \shm_j=\shd_{X_j}/\tw\shi_j \quad (j=1,2).
\eneqn
Note that $\shm_1\simeq\sho_{X_1}$, the de Rham complex on $X_1$.
Then
\eqn
\rhom[\shd_X](\shm_\shi,\shm_\shi)&\simeq&\rhom[\shd_{X_2}](\shm_2,\rhom[\shd_{X_1}](\shm_1,\shm_\shi)).
\eneqn
Here we write for short $\shd_{X_j}$ instead of $\opb{p_j}\shd_{X_j}$ where $p_j\cl X\to X_j$ is the projection, and similarly with the $\shm_j$'s.
An easy calculation gives
\eqn
&&\rhom[\shd_{X_1}](\shm_1,\shm_\shi)\simeq \shm_2.
\eneqn
Hence, we are reduced to treat $\shm_2$ in which case $\irr(\shi_2)=\rk(\shi_2)$. This completes the proof of (A) since
$\irr(\shi_2)=\irr(\shi)$.

\vspace{0.3ex}\noindent
(B) Let us prove that $H^r(\shd_\shi)\not=0$.
By the same argument as in (A) we may assume that $\cork(\shi)=0$. Since $\{v_1,\dots,v_r\}$ is a regular sequence, an easy calculation gives
\eqn
&&H^r(\shd_\shi)\simeq \shd_X/\sum_{j=1}^r (v_j\cdot\shd_X+\shd_X\cdot v_j).
\eneqn
We may assume that $X=\C^n$ and all $v_j$'s vanish at $0$. Hence, we are reduced to prove that the equation
\eq\label{eq:1=A+B}
&&1=\sum_{j=1}^r (v_j\cdot A_j+B_j\cdot v_j)
\eneq
has no solutions $A_j,B_j\in\shd_X$. Let us argue by contradiction and apply the right-hand side of~\eqref{eq:1=A+B} to the holomorphic function $1$. Since $v_j(1)=0$,  we get:
\eqn
(\sum_{j=1}^r (v_j\cdot A_j+B_j\cdot v_j))(1)=\sum_{j=1}^r (v_j\cdot A_j)(1)=\sum_{j=1}^r (v_j\cdot A_j)(1).
\eneqn
Since $v_j$ vanishes at $0$,  the differential operator $v_j\cdot A_j$ also vanishes at $0$ (meaning that, if one chooses a local coordinate system, all coefficients of this operator will vanish at $0$). Therefore, $(v_j\cdot A_j)(1)=0$.
\end{proof}

 We examine some examples in which we calculate $H^i(\shd_\shi)$ for $1\leqslant i \leqslant r$.

 \begin{example}
Let us particularize Theorem~\ref{th:1} in a simple situation.

Let $X=\C^2$  endowed with coordinates $(x,y)$. Consider the vector field
$v=x\partial_x+y\partial_y$ and let $\shi$ be the Lie subalgebra of
$\Theta_X$ generated
by this vector field.
If $f$ is a section of $\sho_X$, then $v(f)=0$ implies that $f$ is
homogeneous of degree $0$.

We have $X=X_0\sqcup X_1$ where $X_1=\{0\}$ and $\shi$ has rank $1$ on
$X_0$, $0$ on $X_1$.
The leaves of $X_0$ are the complex curves $\{(x,y);x^2+y^2=c\}$ ($c\in\C,c\not=0$).

As already mentioned, the restriction of $\shd_\shi$ to $X_0$ is concentrated
in degree $0$. Let us calculate $H^1(\shd_\shi)$. The module
$\shm_\shi=\shd_X/\shd_X\cdot v$ is represented by the complex
\eqn
&&0\to\shd_X^{-1}  \to[v]\shd_X^0\to 0
\eneqn
in which $\shd_X^{i}=\shd_X$ ($i=0,-1$), $v$ operates on the right and $\shd_X^0$  is in
degree $0$. Therefore, $\rhom[\shd_X](\shm_\shi,\shd_X)$ is
represented by the same complex where now $v$ acts on the left and
$\shd^0_X$ is in degree $1$. It follows that
\eqn
&&H^1(\shd_\shi)\simeq\shd_X/(v\cdot\shd_X+\shd_X\cdot v).
\eneqn
Since $1\notin v\cdot\shd_X+\shd_X\cdot v$, we deduce that $H^1(\shd_\shi)\not=0$.
\end{example}

  \begin{example}
 Assume that
 \eq
\left\{\parbox{60ex}{
(i) $\shi$ is of rank $2$,\\
(ii) $\shi$ is a Lie abelian subalgebra of $\Theta_X$,\\
(iii) $\cork(\shi)=0$.
}\right.
\eneq
By \cite[Rem.~4.3]{Su90} and hypothesis~(ii) there is a generator system $\{v_1, v_2 \}$ of $\shi$ satisfying hypothesis~\eqref{hyp:2}. Therefore, Theorem~\ref{th:1}  implies $H^2(\shd_\shi)\not=0$.
By hypothesis $(iii)$  $\{v_1, v_2 \}$ vanish at $0$ hence $1\notin v_{1}\cdot\shd_X+\shd_X\cdot v_{1}$ and we have also $H^1(\shd_\shi)\not=0$.
\end{example}

{\bf Some questions}
\begin{itemize}
\item
Is it possible to weaken Hypothesis~\ref{hyp:2}?
\item
What is the geometric meaning of the sequence $\shd$-$\irr(\shi)$?
\end{itemize}

\providecommand{\bysame}{\leavevmode\hbox to3em{\hrulefill}\thinspace}

\end{document}